\documentclass[11pt]{amsart}
\usepackage{calc,amssymb,amsthm,amsmath,fullpage}
\RequirePackage[dvipsnames,usenames]{xcolor}
\usepackage{hyperref}
\usepackage{graphicx}
\usepackage{amssymb,amsthm,amsmath,epic,eepic,fullpage, srcltx}
\usepackage[all]{xy}
\hypersetup{
bookmarks,
bookmarksdepth=3,
bookmarksopen,
bookmarksnumbered,
pdfstartview=FitH,
colorlinks,backref,hyperindex,
linkcolor=Sepia,
anchorcolor=BurntOrange,
citecolor=MidnightBlue,
citecolor=OliveGreen,
filecolor=BlueViolet,
menucolor=Yellow,
urlcolor=OliveGreen
}
\usepackage{multicol}
\usepackage{xspace}
\usepackage{rotating}
\usepackage{mabliautoref}
\usepackage{colonequals}

\newcommand{\SN}{{\mathrm{SN}}}
\newcommand{\Norm}{{\mathrm{N}}}

\input kmacros3.sty
\input xy
\usepackage{stmaryrd}

\usepackage{verbatim}
\usepackage{enumerate}
\begin{document}
\title{{Seminormalization} package for \emph{Macaulay2}}

\author{Bernard Serbinowski}
\date{\today}							
\address{Department of Computer Science, Vanderbilt University, PMB 351679, 2301 Vanderbilt Place, Nashville, TN 37235, USA}
\email{bserbinowski@gmail.com}

\author{Karl Schwede}
\date{\today}
\address{Department of Mathematics, University of Utah, 155 S 1400 E Room 233, Salt Lake City, UT, 84112}
\email{schwede@math.utah.edu}

\begin{abstract}
This note describes a package for computing seminormalization of rings within Macaulay2.
\end{abstract}



\thanks{The first named author was supported in part by NSF CAREER Grant DMS \#1252860/1501102 and NSF grant \#1801849.}
\thanks{The second named author was supported in part by NSF CAREER Grant DMS \#1252860/1501102.}
\maketitle

\section{Introduction}

Given a reduced Noetherian ring $R$, between $R$ and its normalization $R^{\Norm}$, there is the seminormalization $R^{\SN}$.  In this paper we discuss an implementation of a seminormalization algorithm within Macaulay2.  The ring $R$ is called \emph{seminormal} if every finite birational extension $R \subseteq S$ that induces a bijection on primes and an isomorphism of residue fields, is in fact an isomorphism see for instance \cite{TraversoPicardGroup,GrecoTraversoSeminormal,LeahyVitulliSeminormalRingsWeaklyNormalVars}.  In particular, a \emph{cusp} is not seminormal, since its normalization map is a bijection on points and an isomorphism of residue fields.

Let us delve a little deeper into non-normal rings, in a way that will help explain the algorithm.  Suppose that $R$ is as above with normalization $R^{\Norm}$.  The conductor $\frc \subseteq R \subseteq R^{\Norm}$ is defined to be $\Ann_R(R^{\Norm} / R)$.  It is an ideal in both $R$ and $R^{\Norm}$ which defines the locus where $R$ is not normal.  In this situation, $R$ is always the pullback of the following diagram:
\begin{equation}
\label{eq.ConductorPullback}
\xymatrix{
R^{\Norm}/\frc & \ar@{_{(}->}[l] R/\frc \\
R^{\Norm} \ar@{->>}[u] & \ar@{.>}[l] \ar@{.>}[u] R
}
\end{equation}
Or in other words
\[
R \cong \{ (x, y + \frc) \in R^{\Norm} \times R/\frc \;|\; x + \frc = y + \frc \}.
\]
Since the pullback of this diagram dualizes to the pushout when taking $\Spec$, we can interpret $\Spec R$ as a quotient of $\Spec R^{\Norm}$ where certain points are identified (or have their residue fields shrunk) and certain tangent spaces are glued or otherwise annihilated (the latter owing to the \emph{scheme structure} of $R^{\Norm}/\frc$).  For additional discussion, see for instance \cite{SchwedeMathOverflowNormalization}.  In view of this construction, a ring is seminormal if its non-normality is due \emph{only} to gluing of points.  In other words, a seminormal ring is one where there is no undue identification of tangent spaces.

This idea leads us to our algorithm for seminormalizing, which is the topic of the next section.

\subsection*{Acknowledgements}  The authors thank Neil Epstein and Claudiu Raicu for stimulating discussions and in particular to Claudiu Raicu for writing and then improving the {\tt PushForward} package \cite{RaicuPushForward} in ways that helped the development of this package.

\section{Structure of the algorithm}

The idea of the algorithm is to perform the pullback from \autoref{eq.ConductorPullback} but instead of modding out by $\frc$, we want to remove unnecessary tangent space identification.  A simple option would be to form the pullback $S$ of the diagram:
\begin{equation}
\label{eq.ConductorPullback}
\xymatrix{
R^{\Norm}/\sqrt{\frc R^{\Norm}} & \ar@{_{(}->}[l] R/\sqrt{\frc R} \\
R^{\Norm} \ar@{->>}[u] & \ar@{.>}[l] \ar@{.>}[u] S
}
\end{equation}
but this is not the seminormalization of $R$ since $R/\sqrt{\frc R}$ could itself have undue gluing of tangent spaces.  An easy way to get around this is to seminormalize $R/\sqrt{\frc R}$ (which has lower dimension than $R$, and so a recursive algorithm can apply), but the seminormalization $(R/\sqrt{\frc R})^{\SN}$ does not necessarily map to $R^{\Norm}/\sqrt{\frc R^{\Norm}}$ (since that is not necessarily seminormal).  We could also seminormalize $R^{\Norm}/\sqrt{\frc R^{\Norm}}$, but this led to implementation difficulties and so we instead form the intersection:
\[
D = (R/\sqrt{\frc R})^{\SN} \bigcap R^{\Norm}/\sqrt{\frc R^{\Norm}}
\]
where the intersection takes place in the total ring of fractions of $R^{\Norm}/\sqrt{\frc R^{\Norm}}$.
Then we pullback the diagram
\[
\xymatrix{
R^{\Norm}/\sqrt{\frc R^{\Norm}}  & \ar@{_{(}->}[l] D\\
R^{\Norm} \ar@{->>}[u] & C \ar@{.>}[u] \ar@{.>}[l]
}
\]
\begin{theorem}
\label{thm.PullbackIsSeminormalization}
The ring $C$ in the above diagram is the seminormalization of $R$.
\end{theorem}
\begin{proof}
We first notice that there is a diagram:
\[
\xymatrix{
(R^{\Norm}/\sqrt{\frc R^{\Norm}})^{\SN}  & \ar@{_{(}->}[l] (R/\sqrt{\frc R})^{\SN}\\
R^{\Norm} \ar[u] & R^{\SN} \ar@{.>}[u]_{\beta} \ar@{.>}[l]
}
\]
coming from the functoriality of seminormalization.  Since the image of $R^{\SN}$ maps into $(R^{\Norm}/\sqrt{\frc R^{\Norm}})$, we see that $\beta(R^{\SN}) \subseteq D$.  Hence by the universal property of pullback, there is a map $R^{\SN} \to C$.  On the other hand, it follows from \cite{FerrandConducteurEtPincement} (\cf \cite{SchwedeGluing}), that the map $R^{\SN} \to C$ is induces bijection on points of $\Spec$ and induces an isomorphism of residue fields.  Indeed recall that the pullback of a diagram $\{A \to A/I \leftarrow B \}$ replaces the closed subscheme $V(I) \subseteq \Spec A$ with $\Spec B$, residue fields and all.  Since $C \hookrightarrow R^{\Norm}$, we have that $R^{\SN} \to C$ is also birational and so $R^{\SN} \to C$ is an isomorphism since $R^{\SN}$ is seminormal by definition.  This completes the proof.
\end{proof}

\section{Implementation of the algorithm}
We describe the main algorithm first, and then describe the strategy of some of the component functions individually.

\subsection{The main seminormalization algorithm}

As the algorithm is recursive, the first thing we do is check whether or not we are finished. If the Krull dimension of the ring is 0 or if the ring is normal we return the ring unchanged as it has already been seminormalized (note we assume that the ring the function is given is reduced).  This ensures that the process will end since at each step of the induction, the dimension will drop.

Assuming the ring is not already seminormal, then we create a map from the input ring to its normalization:
\[
\phi : R \to R^{\Norm}.
\]
Note, we do not use the core normalization function {\tt integralClosure} as that does not work correctly on non-domains (even if the ring is reduced).  Instead we call a function {\tt betterNormalizationMap} (which eventually calls {\tt integralClosure} on various quotient rings of $R$), see \autoref{sec.BetterNormalization}.

We then compute the conductor of this map, which in previous sections we called $\frc$.  We then take the radical of this ideal in both $R$ and $R^{\Norm}$ and now we can form the following diagram.
\[
\xymatrix{
R^{\Norm}/\sqrt{\frc R^{\Norm}}  & \ar@{_{(}->}[l] R/\sqrt{\frc R}\\
R^{\Norm} \ar@{->>}[u]^{\phi} &
}
\]
At this point, we make our recursive call and {\tt seminormalize} $R/\sqrt{\frc R}$ (constructing a map $\gamma : R/\sqrt{\frc R} \to R/\sqrt{\frc R}^{\SN}$ in the process).  Finally, we need to construct the ring we called $D$ above, $D = (R/\sqrt{\frc R})^{\SN} \cap R^{\Norm}/\sqrt{\frc R^{\Norm}}$.  This is a bit tricky, and is the subject of \autoref{sec.IntersectingRingExtensions} below.  In the meantime, once we have constructed $D$ (and the map $\psi : D \to R^{\Norm}/\sqrt{\frc R^{\Norm}}$), we can pullback the diagram
\[
\xymatrix{
R^{\Norm}/\sqrt{\frc R^{\Norm}}  & \ar@{_{(}->}[l]_-{\psi} D\\
R^{\Norm} \ar@{->>}[u]^{\phi} &C \ar@{.>}[u] \ar@{.>}[l]
}
\]
to obtain $C$, which we have already verified is the seminormalization in \autoref{thm.PullbackIsSeminormalization}.  Note we perform this pullback by using the package {\tt Pullback.m2} \cite{EllingsonSchwedePullback}, which requires the map $\phi$ above to be surjective.

\subsection{Normalization of reduced rings}
\label{sec.BetterNormalization}

As mentioned above, this package includes a function {\tt betterNormalizationMap} which computes the normalization of reduced rings.  The strategy is as follows.

\begin{description}
\item[Step 1]  Compute the minimal primes $\{\frq_i\}$ of $R$.
\item[Step 2]  Compute the normalizations of the $R/\frq_i$.  Note the function {\tt betterNormalizationMap} has an option {\tt Strategy} which is passed to the {\tt integralClosure} calls at this step.
\item[Step 3]  Construct the product of the normalized rings $R^{\Norm} = \prod_i  (R/\frq_i)^{\Norm}$.
\item[Step 4]  Construct the map from $R$ to $R^{\Norm}$.
\end{description}

Step 3, constructing the product of normalized rings, is achieved by calling a function {\tt ringProduct} which computes a product of a list of rings defined over the same base ring (\ie, defined over {\tt QQ}).  This returns the product of rings as well as the list of orthogonal idempotents defining each ring.  It also returns a list of lists showing what variables from our original rings become in the product.  We hope that this functionality of taking products of rings may be useful in other contexts besides computing normalizations.

Step 4 is the most involved.  We first construct various maps $(R/\frq_i)^{\Norm} \to R^{\Norm}$, notice this is not a \emph{real} ring map we want, we are using it to keep track of where variables of the $(R/\frq_i)^{\Norm}$ go.  We then compose with $R \to (R/\frq_i)^{\Norm}$ to obtain various different maps $R \to R^{\Norm}$.  Finally, we sum over all these maps (multiplying by our orthogonal idempotents as appropriate) to obtain our normalization map $R \to R^{\Norm}$.

\subsection{Intersecting the seminormalization and another ring extension}
\label{sec.IntersectingRingExtensions}

At a key point in our algorithm, we have two extensions of $A = R/{\sqrt{\frc R}}$, first the seminormalization $A^{\SN} = (R/{\sqrt{\frc R}})^{\SN}$ and second, the finite extension to $B = R^N/\sqrt{\frc R^N}$.  We need to form an intersection of these two extensions.  We do this by using the function {\tt intersectSeminormalizationAndExtension} which computes exactly this intersection (and a ring map from our base ring to the intersection).

To do this, first we find a reduced ring $O$ containing both of these extensions whose minimal primes are in bijection with the minimal primes of the ring we called $B$.  This is done via the function {\tt findOverring} which essentially tensors the two extension rings together and then drops any unnecessary minimal primes.  Note we do not have to worry about how the the seminormalization embeds into this overring by uniqueness properties of elements of the seminormalization \cite{SwanSeminormality}. Once we have the overring $O$, we form the exact sequence (making liberal use of the {\tt PushFwd} package)
\[
A^{\SN} \oplus B \xrightarrow{(a,b) \mapsto a-b} O \to 0
\]
and computing the kernel $K$.  At this point, $K$ is the desired intersection ring, but Macaulay2 only understands it as a module.  However, we can take the module generators of $K$, map them into $B$, and consider the ring they generate.  This is our desired ring.

\subsection{Variable naming conventions}
\label{sec.NamingConventions}

One issue we ran into when calling a recursive function that produces new rings is that there can be numerous collisions of variable names.  Because of this, we made the following conventions in variable naming internally.

Our internal recursion calls give variables numbers based on recursion depth and where the renaming is occurring. Each time a recursive call occurs, we increment a counter by 2. This allows us to trace where variables were renamed. Odd numbers indicate renaming occurred within {\tt intersectSeminormalizationAndExtension}, while even numbers indicate the renaming happened within the main recursive method. In addition to this, we also add variables named e$N$ where $N$ is replaced by an integer. These variables are used to ensure that the seminormal property is satisfied for the ring.

However, none of this is visible in the outputted ring, as by default \emph{all the variables} of the output ring will have been renamed as ${\tt Yy}_N$ where $N$ varies.  If you do not want to use {\tt Yy}, you may instead supply your own variable name via the {\tt Variable $\Rightarrow$ X} option when calling Seminormalize. Here {\tt X} must be a valid symbol. It is important to note that you cannot use a symbol that overlaps with an existing variable as this will cause errors.  The output does include a map from the original ring to the seminormalization.  See the examples below.

\section{Examples}

The function {\tt seminormalize} returns a list of three things.  First it returns the ring $R^{\SN}$, then it returns the ring map $R \to R^{\SN}$ and finally it returns the ring map $R^{\SN} \to R^{\Norm}$.

We begin by seminormalizing the cusp, in this case the seminormalization is the normalization.

\begin{verbatim}
i1 : loadPackage "Seminormalization"
o1 = Seminormalization
o1 : Package
i2 : R = QQ[x,y]/ideal(y^2-x^3);
i3 : seminormalizedList = (seminormalize(R));
i4 : seminormalizedList#0

                QQ[Yy , Yy , Yy ]
                     0    1    2
o4 = ---------------------------------------
        2                        2
     (Yy  - Yy , Yy Yy  - Yy , Yy  - Yy Yy )
        2     1    1  2     0    1     0  2

o4 : QuotientRing
i5 : prune seminormalizedList#0
o5 = QQ[Yy ]
          2
o5 : PolynomialRing
i6 : seminormalizedList#1
                    QQ[Yy , Yy , Yy ]
                         0    1    2
o6 = map(---------------------------------------,R,{Yy , Yy })
            2                        2                1    0
         (Yy  - Yy , Yy Yy  - Yy , Yy  - Yy Yy )
            2     1    1  2     0    1     0  2

                        QQ[Yy , Yy , Yy ]
                             0    1    2
o6 : RingMap --------------------------------------- <--- R
                2                        2
             (Yy  - Yy , Yy Yy  - Yy , Yy  - Yy Yy )
                2     1    1  2     0    1     0  2
i7 : isSeminormal(R)
o7 = false
\end{verbatim}

Next, we seminormalize four lines through the origin in $\bA^2$.  This should produce something isomorphic to 4 coordinate axes in $\bA^4$, which it does.

\begin{verbatim}
i2 : R = QQ[x,y]/ideal(x*y*(x^2-y^2));
i3 : seminormalizedList = seminormalize(R);
i4 : seminormalizedList#0
                  QQ[Yy , Yy , Yy , Yy ]
                       0    1    2    3
o4 = ------------------------------------------------
     (Yy Yy , Yy Yy , Yy Yy , Yy Yy , Yy Yy , Yy Yy )
        2  3    1  3    0  3    1  2    0  2    0  1
o4 : QuotientRing
\end{verbatim}

The following example of Greco and Traverso is a seminormal ring whose prime spectrum has an irreducible component that is not seminormal \cite{GrecoTraversoSeminormal}.
\begin{verbatim}
i2 : B = ZZ/11[x,y,u,v,e,f];
i3 : I = intersect(ideal(u,v,e-1,f),ideal(x,y,e,f-1));
i4 : A = B/I;
i5 : E = ZZ/11[z1, z2, z3, z4, z5];
i6 : h = map(A, E, {x^3+u, x^2+v, y, u^2-v^3, x*y});
i7 : J = ker h;
i8 : D = E/J;
i9 : isSeminormal(D) --this should be seminormal
o9 = true
i10 : JJ = preimage(h, ideal(sub(f,A)));
i11 : D2 = E/(trim(JJ + J));
i12 : isSeminormal(D2) --this should not be seminormal
o12 = false
\end{verbatim}
Finally, we verify the seminormality of a seminormal ring that is not weakly normal.
\begin{verbatim}
i2 : R = ZZ/2[t, x, y]/ideal(x^2 - t*y^2);
i3 : isSeminormal(R)
o3 = true
\end{verbatim}

\bibliographystyle{skalpha}
\bibliography{MainBib}

\end{document}